\documentclass[12pt,a4paper]{amsart}

\usepackage[usenames,dvipsnames,svgnames,table]{xcolor} 
\usepackage{amssymb,amsfonts,amsrefs}
\usepackage{amsmath,amsthm}

\usepackage{marginnote}

\usepackage{hyperref}
\usepackage[active]{srcltx}		
\usepackage{pdfsync}					
\usepackage[colorinlistoftodos]{todonotes}
\newcommand\R{{\mathbb{R}}}

\newcommand\N{{\mathbb{N}}}

\newtheorem{theorem}{Theorem}[section]
\newtheorem{proposition}[theorem]{Proposition}
\newtheorem{lemma}[theorem]{Lemma}

\theoremstyle{definition}

\theoremstyle{remark}
\newtheorem{remark}[theorem]{Remark}
\numberwithin{equation}{section}

\begin{document}

\title[Self-similar solutions to Oberbeck--Boussinesq  system]
{Large self-similar solutions\\ to Oberbeck--Boussinesq  system\\
 with Newtonian gravitational field}

\author{Lorenzo Brandolese}
\address{L. Brandolese: Universit\'e Claude Bernard Lyon 1, CNRS UMR 5208, Institut Camille Jordan, 43 blvd. du 11 novembre 1918, F-69622 Villeurbanne cedex, France\\
\href{https://orcid.org/0000-0002-3045-0016}{orcid.org/0000-0002-3045-0016}}
\email{brandolese{@}math.univ-lyon1.fr}
\urladdr{http://math.univ-lyon1.fr/$\sim$brandolese}

\author{Grzegorz Karch}
\address{G. Karch: Instytut Matematyczny, Uniwersytet Wroc\l awski pl.
Grunwaldzki 2, 50-384 Wroc\l aw, Poland\\
\href{https://orcid.org/0000-0001-9390-5578}{orcid.org/0000-0001-9390-5578}}
\email{karch{@}math.uni.wroc.pl}
\urladdr{http://www.math.uni.wroc.pl/$\sim$karch}

\begin{abstract}
The Navier-Stokes system for an incompressible fluid coupled with the equation for a heat transfer is considered in the whole three-dimensional space. 
This system is invariant under a suitable scaling.
Using the Leray--Schauder theorem and compactness arguments, we construct self-similar solutions to this system  without any smallness assumptions imposed on homogeneous initial conditions.
\end{abstract}



\thanks{The research was supported by the Arqus European University Alliance.}
\date{\today}

\maketitle


\bigskip
Published in Journal of Differential Equations,
Volume 424 (2025), Pages 139-158.

\section{Introduction}

The Oberbeck–Boussinesq system is a mathematical
model of stratified flows, where the fluid is assumed to be incompressible and 
convecting, due to the presence of buoyancy forces arising from temperature variations inside the fluid. The equation of the fuid motion is coupled with a diffusive equation for the deviation of the temperature from its equilibrium value. 
The resulting system has the following form
\begin{equation}\label{B}
\begin{split}
 \partial_t u +u\cdot\nabla u+\nabla p&=\Delta u+\theta \nabla G +f,\quad x\in \R^3,\; t\in \R_{+} \\
 \nabla\cdot u&=0, \\
  \partial_t \theta +u \cdot \nabla \theta &=  \Delta \theta, 
  \end{split}
\end{equation}
with the unknown  fluid velocity $u = u(x,t)$, the temperature $\theta = \theta(x,t)$ and the pressure $p=p(x,t)$.
In first equation of system \eqref{B}, the symbol $\nabla G$ denotes a gravitational force acting on the fluid
and $f=f(x,t)$ is  a given external force. 
The equations in  system \eqref{B} should contain important  physical parameters such as  
is the viscosity coefficient,  the heat conductivity coefficient, the fluid density, the reference temperature, the specific heat at constant pressure,  the coefficient of thermal expansion of the fluid. However, since these physical constants do not play any role in this work, we put all of them equal to one, for simplicity of the exposition.

It is customary to take $\nabla G = g(0,0,1)$, where the constant $g$ represents the acceleration rate caused by Earth's gravity. This is a reasonable approximation provided that the fluid occupies a bounded domain, where the gravitational field can be taken constant.
Moreover, several authors addressed various problems of mathematical interest for  system \eqref{B} on the whole space $\R^3$,
still adopting the choice of a constant gravity field $\nabla G = g(0,0,1)$:
see, {\it e.g}, \cites{CD80,DP08,KP08,BS12} and references therein.

{\color{blue}
But the choice of a constant gravity field 
brings buoyancy effects on the fluid 
that, in unbounded domains,
can lead to physically unrealistic situations. For example, as shown in \cite{BS12}, in the whole space~$\R^3$,
Oberbeck-Boussinesq flows with initially finite energy,  in the absence of
any external forcing, in general feature energy growth, with the kinetic energy eventually growing 
arbitrarily large for long times.}

As noticed by Feireisl and Schonbek \cite{FS12},
for system \eqref{B} in the whole space $\R^3$, from a physical point of view 
it is more appropriate to consider the gravitational potential of the following form 
$$ G(x) =-\int_{\R^3} \frac{1}{ |x - y|}m(y)\, dy, $$ where $m$ denotes the mass density of the object acting on the fluid by means of gravitation.
Moreover, if the size of the object is negligible, one is led to choosing $G(x)=-|x|^{-1}$. 
In fact, the Oberbeck–Boussinesq approximation \eqref{B}
with this singular gravitational force
can be rigorously derived as a singular limit
of the full Navier-Stokes–Fourier system  with suitable boundary conditions and with the Mach
and Froude numbers tending to zero and when the family of domains on which the primitive problems
are stated converges to the whole space $\R^3$, see \cites{FS12,WK17} and the references therein.  

For this reason, the present work is devoted to system \eqref{B} with singular gravitational force
$\nabla G(x)=\nabla |x|^{-1}$. We dropped the minus in front of the gradient because the sign
will not affect our results.

The Oberbeck-Boussinesq model
enjoys the following scaling property: 
 if $(u,  \theta)$ is a solution of system \eqref{B} with  
\begin{equation}\label{Gf}
G(x) = \frac{1}{|x|} \quad \text{and}\quad   f(x,t) = \frac{1}{{(\sqrt{2t})}^{3}}F\left(\frac{x}{\sqrt{2t}}\right),
\end{equation}
then 
\begin{equation}
 u_\lambda(x,t):=\lambda u(\lambda x,\lambda^2 t), 
 \qquad \theta_\lambda(x,t):=\lambda\theta(\lambda x,\lambda^2 t)
 \end{equation}
 is also a solution of the same system for each $\lambda>0$.
A {\it self-similar solution} to system \eqref{B}-\eqref{Gf} is, by definition, a solution which is left invariant by this rescaling:
$
(u,\theta)= (u_\lambda, \theta_\lambda)
$
for every $\lambda>0$.
Equivalently, by choosing $\lambda=1/\sqrt{2t}$,
self-similar solutions are those that can be written in the form
\begin{equation} \label{ut:self}
u(x,t)=\frac{1}{\sqrt{2t}}U\Bigl(\frac{x}{\sqrt{2t}}\Bigr),
\qquad
\theta(x,t)=\frac{1}{\sqrt{2t}}\Theta\Bigl(\frac{x}{\sqrt{2t}}\Bigr),
\end{equation}
with
$U(x)=u(x,1/2)$  and $\Theta(x)=\theta(x,1/2)$.

If system \eqref{B} is supplemented with an initial condition
\begin{equation}\label{ini}
u(x,0)=u_0(x), \qquad \theta(x,0) =\theta_0(x),
\end{equation}
in the case of a self-similar solution, it has to be also invariant under the scaling 
$(u_{0,\lambda},\theta_{0,\lambda})(x)=\lambda(u_0,\theta_0)(\lambda x)$, which means 
that these are homogeneous function of degree $-1$.
In this work, we construct self-similar solutions of system \eqref{B}-\eqref{Gf}
with arbitrary, not necessarily small initial conditions \eqref{ini} which are homogeneous of degree $-1$ and essentially bounded on the unit sphere of $\R^3$.
The main result of this work is stated in the following theorem.

\begin{theorem}
\label{th:main}
Let $(u_0,\theta_0)\in {\bf L}^\infty_{\rm loc}(\R^3\backslash\{0\})$ be homogeneous of degree $-1$, with $\mbox{$\nabla\cdot u_0$}=0$.
Let the external force $f(x,t)$ be of the form as in \eqref{Gf} with the profile 
{\color{blue}$F\in H^{-1}(\R^3)^3$}.
Then there exists a self-similar solution $(u,p)$ 
of system \eqref{B}-\eqref{Gf}.
This solution has the following properties:
\begin{itemize}
    \item $(u,\theta)\in C_w([0,\infty),{\bf L}^{3,\infty}(\R^3))$;
    \item  for some constants $c,c'\ge0$ and all $t>0$,
\begin{equation}
\label{timeest}
\begin{split}
&\|u(t)-e^{t\Delta}u_0\|_2+\|\theta(t)-e^{t\Delta}\theta_0\|_2=ct^{1/4},\\
& \|\nabla u(t)-\nabla e^{t\Delta}u_0\|_2+\|\nabla\theta(t)-\nabla e^{t\Delta}\theta_0\|_2
= c't^{-1/4}.
\end{split}
\end{equation}
\end{itemize}
\end{theorem}


Theorem \ref{th:main}
 extends to the Oberbeck–Boussinesq system \eqref{B}-\eqref{Gf} the 
 well-known existence result of large
({\it i.e.}~with no size restriction on the initial data) forward self-similar solutions to the Navier--Stokes equations, first established by Jia and \v{S}ver\'ak~\cite{JS14}.

Let us briefly review related results on self-similar solutions to the Navier-Stokes system and the Oberbeck–Boussinesq system.
If initial data are sufficiently small, unique mild solutions to the Cauchy problem either for the Navier-Stokes system or the Oberbeck–Boussinesq system can be obtained via the contraction mapping
argument applied to integral formulations of these problems. If the considered space has a scaling-invariant norm and contains homogeneous initial conditions, the uniqueness property ensures that obtained solutions are self-similar.
This approach was first applied to the Cauchy problem for the Navier-Stokes system 
by Giga-Miyakawa \cite{GM89} and refined by Kato \cite{K92} and Cannone-Meyer-Planchon \cites{CMP94,CP96}, see also {\it e.g.} \cites{K99,Y00,C04,B09}
and references therein. 
Methods of constructing small self-similar solutions corresponding to small initial conditions were then applied to other models including the Oberbeck–Boussinesq, see {\it e.g.}~\cites{KP08,AF11,FV10}.

For large initial conditions, the contraction mapping argument no longer works and the question on the existence of large self-similar solutions remained open until the seminal paper by 
Jia and \v{S}ver\'ak~\cite{JS14} who constructed self-similar solutions of the three-dimensional Navier-Stokes system,
supplemented with a homogeneous not necessarily small initial condition which is H\"older continuous outside the origin.
In their construction, they used the theory of local Leray solutions in $L^2_{uloc}$,
the space of uniformly locally $L^2$-functions, developed by 
 Lemari\'e-Rieusset \cite{LR02} and they obtained  
 a local H\"older estimate for local Leray solutions near $t = 0$, assuming minimal control
of initial data. That estimate enables them to prove {\it a priori} estimates of
self-similar solutions, and then to show their existence by the Leray-Schauder degree theorem. 
The results of Jia and \v{S}ver\'ak~\cite{JS14} were then extended either to discretely self-similar solutions 
of the Navier-Stokes system in the whole or a half-space \cites{T14, KT16,BT17,BT18,BT19, BP23}
 or to the fractional Navier-Stokes system \cite{LMZ19}
 or to the MHD  system~\cite{L19}.

Now let us describe the strategy of proving Theorem \ref{th:main}.
Substituting expressions \eqref{ut:self} into system \eqref{B}-\eqref{Gf} allows us to eliminate 
 the time variable and 
we are led to construct a solution $(U,P, \Theta)$ to the following elliptic system
\begin{equation}
\label{BSS}
\begin{split}
-\Delta U -U-(x\cdot\nabla) U+(U\cdot\nabla)U+\nabla P&=\Theta\,\nabla(|\cdot|^{-1})+F,\\
\nabla\cdot U&=0,\\
-\Delta \Theta-\Theta-(x\cdot\nabla)\Theta+\nabla(\Theta U)&=0.
\end{split}
\end{equation}
In this construction, we use  solutions to the heat equation given as convolutions of the Gauss-Weierstrass kernel with homogeneous initial data which are themselves of
self-similar form: in particular, by our assumptions on $u_0$ and $\theta_0$, we have
\begin{equation}\label{heat}
e^{t\Delta}u_0(x)=\frac{1}{\sqrt{2t}}U_0\left(\frac{x}{\sqrt{2t}}\right) 
\quad \text{and} \quad
e^{t\Delta}\theta_0(x)=\frac{1}{\sqrt{2t}}\Theta_0\left(\frac{x}{\sqrt{2t}}\right),
\end{equation}
where the self-similar profiles 
\begin{equation}
\label{UT0}
U_0:=e^{\Delta/2}u_0
\quad\text{and}\quad
\Theta_0:=e^{\Delta/2}\theta_0
\end{equation}
have properties recalled below in Proposition \ref{prop:heat}.
Then, rather than studying directly system~\eqref{BSS}, we will consider the perturbations
\[
V=U-U_0
\quad\text{and}\quad
\Psi=\Theta-\Theta_0,
\]
which satisfy the elliptic system
\begin{equation}
\label{DSS}
\begin{split}
-\Delta V-V-(x\cdot\nabla)V+(V+U_0)\cdot\nabla(V+U_0)+\nabla P&=(\Psi+\Theta_0)\,\nabla(|\cdot|^{-1})+F,\\
\nabla \cdot V&=0,\\
-\Delta \Psi-\Psi-(x\cdot\nabla)\Psi+\nabla\cdot\bigl((\Psi+\Theta_0)(V+U_0)\bigr)&=0.
\end{split}
\end{equation}
In the next section, we will construct solutions of system \eqref{DSS} in the Sobolev space $H^1(\R^3)^4$, 
see~Theorem~\ref{prop:exi-R3} below, and 
we deduce Theorem~\ref{th:main} as a direct corollary.
Our strategy of studying system \eqref{DSS} is closely inspired by  the paper of Korobkov and Tsai \cite{KT16}, where they established a similar result for the Navier--Stokes equations in the half-space.
In that approach, we first solve system \eqref{DSS} supplemented with the Dirichlet boundary condition in a ball by using the Leray--Schauder theorem. Then, we obtain a solution in the whole space by passing with the radius of the ball to infinity and using an $H^1$-estimate of the sequence of solutions which is independent of the radius of the ball.
The singular nature of the forcing term arising from the temperature variations and the coupling both bring a few new technical difficulties. 
We overcome them by means of an approximation procedure and suitable {\it a priori} estimates, whose derivation does not appear to be so standard 
(see, {\it e.g.}~the contradiction argument contained in the
proofs of Propositions~\ref{prop:1ae} and \ref{prop:invading}).
Assuming that the components of the initial data are $L^\infty_{\rm loc}(\R^3\backslash\{0\})$ and not, e.g., only $L^2_{\rm uloc}$, considerably simplifies the presentation. This makes possible to provide a proof which is more elementary than that presented in~\cites{JS14,LR16} in the case of the Navier--Stokes equations with external forces, despite the Oberbeck--Boussinesq \eqref{B}-\eqref{Gf} system being
more general.

As the usual practice, we skip any comment about the pressure in this introduction, because it disappears in the weak formulation 
of system \eqref{B}
and because it can be obtained in the well-known way 
by applying the divergence operator to first equation in system \eqref{B}. 
Here, we only mention that the  pressure 
corresponding to the self-similar solution constructed in Theorem \ref{th:main}
is self-similar of the form 
$
p(x,t)={(2t)^{-1}}P\Bigl({x}/{\sqrt{2t}}\Bigr).
$

{\color{blue}The result of Theorem~\ref{th:main} remains true also when an additional term in the equation of the temperature, of the form $\widetilde f(\cdot,t)=(\sqrt{2t})^{-1}\widetilde F(\cdot/\sqrt{2t})$, with $\widetilde F\in H^{-1}(\R^3)$, is present. Such a term would arise when an external self-similar heating source acts on the fluid.
A similar result holds also for the Boussinesq system in the half-space, as it is the case for
the construction of self-similar solutions for the Navier--Stokes system made in \cite{KT16}.

We also point out that, after the first version of this paper was made available in arXiv, an alternative construction of self-similar solutions for the Oberbeck-Boussinesq system using a different method, effective also for proving the existence of discretely self-similar solutions, was proposed by Tsai in~\cite{Tsai-preprint}.
In both Tsai's and  our approaches, the special structure of the Newtonian gravitational field plays an important role
in establishing the relevant {\it a priori} estimates.
For this reason, our result does not seem to easily extend to the case of a constant gravitational field,
as considered in  \cites{DP08, KP08, BS12}, or to more general gravitational fields, as in
\cite{CD80}.

The arXiv version of this paper also inspired the author of \cite{YY24} to construct forward self-similar solutions
to the three-dimensional Magnetohydrodynamics-Boussinesq system
with Newtonian gravitational field.
}

\subsection*{Notations}
The symbol $\|\cdot\|_p$ denotes the usual Lebesgue $L^p$-norm. The space $L^{p,q}(\Omega)$
are the Lorentz spaces.
If $\Omega$ is a domain of $\R^3$, then we denote by $C^\infty_{0,\sigma}(\Omega)^3$ the space
of smooth, divergence-free vector fields, with support contained in $\Omega$.
We denote by ${\bf H}(\Omega)$ the closure of $C^\infty_{0,\sigma}(\Omega)^3\times C^\infty_0(\Omega)$ in the Sobolev space
$H^1(\Omega)^4$.
In general, we adopt bold symbols for function spaces of $4$-dimensional vector-valued functions. For example, ${\bf L}^p(\Omega)=L^p(\Omega)^3\times L^p(\Omega)$
and ${\bf C}^1(\Omega)=C^1(\Omega)^3\times C^1(\Omega)$. 
The constants in the estimates below are denoted by the same letter $C$, even if they vary from line to line.


\section{Analysis of the perturbed elliptic system}
\label{sec:stat}

We begin by establishing a few simple properties of self-similar solutions to the heat equation that will be useful in the sequel.

\begin{proposition}\label{prop:heat}
    Let $(u_0,\theta_0)\in {\bf L}^\infty_{\rm loc}(\R^3\backslash\{0\})$ be homogeneous of degree $-1$ 
with  $\nabla\cdot u_0=0$ and consider the corresponding self-similar solutions of the heat equation
$e^{t\Delta}u_0$ and $e^{t\Delta}\theta_0$ written in the form \eqref{heat}.
Then, the profiles $U_0$ and $\Theta_0$
 satisfy the following equations
\begin{equation}
\label{eq:ssheat}
\begin{split}
&U_0+x\cdot \nabla U_0+\Delta U_0=0,\qquad
\nabla\cdot U_0=0,\\
&\Theta_0+x\cdot \nabla \Theta_0+\Delta \Theta_0=0
\end{split}
\end{equation}
and the estimates
\begin{equation}
\label{uoto}
\begin{split}
& |U_0(x)|+|\Theta_0(x)| \le C(1+|x|)^{-1},\\
& |\nabla U_0(x)|+|\nabla \Theta_0(x)|\le C(1+|x|)^{-1},
\end{split}
\end{equation}
for all $x\in\R^3$ and  some constant $C>0$ independent on $x$.
\end{proposition}

\begin{proof}
From our assumptions, we deduce that the map
$x\mapsto |x|(|u_0(x)|+|\theta_0(x)|)$ is in $L^\infty(\R^3)$.
Let us denote by $G_t(x)=(4\pi t)^{-3/2}\exp(-|x|^2/(4t))$ the heat kernel.
We do have $U_0=G_{1/2}*u_0$, $\Theta_0=G_{1/2}*\theta_0$
and $\nabla U_0=(\nabla G_{1/2})*u_0$, $\nabla\Theta_0=(\nabla G_{1/2})*\theta_0$.
But $(\nabla G_{1/2},G_{1/2})$ belong to the space ${\bf L}^{3/2,1}(\R^3)$
(and to any other Lorentz space $L^{p,q}(\R^3)$, for $1< p,q\le \infty$). 
As $(u_0,\theta_0)\in {\bf L}^{3,\infty}(\R^3)$, the Young inequality for Lorentz spaces (see
\cite{LR02}*{Ch.~2}) implies that
$(U_0,\Theta_0)\in {\bf L}^\infty(\R^3)$.
From a simple convolution estimate, relying on the fast decay of
$G_{1/2}$ and $\nabla G_{1/2}$ at the spatial infinity,
we easily deduce estimates~\eqref{uoto}.
Moreover, since $u_0$ is divergence-free, it results that $\nabla\cdot U_0=0$.
The two other equations in~\eqref{eq:ssheat} are well known and follow from the scaling invariance of the heat equation.
\end{proof}

\begin{remark}
The second estimate in~\ref{uoto} estimate could be easily improved to $|\nabla U_0(x)|+|\nabla \Theta_0(x)|\le C(1+|x|)^{-2}$, but we will do not need this better decay, as
our arguments just require 
{\color{blue} on $\nabla U_0$ and $\nabla \Theta_0$ that, e.g., ($\nabla U_0,\nabla \Theta_0)$ belong to ${\bf L}^4(\R^3)$.}
\end{remark}

\subsection{A priori estimates for a perturbed elliptic system in bounded domains}
First, we construct solutions to system \eqref{DSS} considered in a bounded domain $\Omega\subset\R^3$ with a smooth boundary.
For technical reasons, we introduce a smooth, bounded function  $\rho\in C_b(\R^3)$ 
which will be used in the next section to cut off the singularity at zero of the potential $|\cdot|^{-1}$.
In view of the application of the Leray-Schauder theorem,
 our first goal is to derive {\it a priori} estimates independent of $\lambda\in [0,1]$ of solutions to the system
\begin{equation}
\label{LSS}
\begin{split}
-\Delta V+\nabla P
&=\lambda\Bigl(V+x\cdot\nabla V+F_0+F_1(V)+(\Psi+\Theta_0)\,\rho\nabla(|\cdot|^{-1})+F\Bigr),\\
\nabla\cdot V&=0,\\
-\Delta\Psi
&=\lambda\Bigl(\Psi+x\cdot\nabla\Psi-\nabla\cdot\bigl((\Psi+\Theta_0)(V+U_0)\bigr)\Bigr),\\
\end{split}
\qquad x\in\Omega
\end{equation}
where, for simplicity of notation, 
we set
\begin{equation*}
F_0:=-U_0\cdot\nabla U_0,\quad\text{and}\quad 
F_1(V):=-(U_0+V)\cdot\nabla V - V\cdot\nabla U_0.
\end{equation*}
and which we 
supplement with the Dirichlet boundary conditions
\begin{equation}\label{LSS:B}
    V=0 \qquad\text{and}\qquad \Psi=0 \qquad\qquad \text{on} \quad \partial\Omega.
\end{equation}
 
\begin{proposition}
\label{prop:1ae}
Let $\Omega$ be a bounded domain in $\R^3$ with a  smooth boundary and
let $\rho\in C_b(\R^3)$.
Assume that 
$(U_0,\Theta_0)\in {\bf L}^\infty(\R^3)$ does satisfy the decay estimates~\eqref{uoto}  
{\color{blue} 
and $F\in H^{-1}(\Omega)^3$.}
Let $\lambda\in[0,1]$ and $(V,\Psi)\in {\bf H}(\Omega)$ be solutions
to problem~\eqref{LSS}-\eqref{LSS:B}. Then there exists a constant 
$C_0=C_0(\Omega,F,\rho,U_0,\Theta_0)$, independent on $\lambda$, such that
\begin{equation}
\label{APB}
\int_\Omega\Bigl(|V|^2+\Psi^2+|\nabla V|^2+|\nabla \Psi|^2\Bigr)
\le C_0. 
\end{equation}
\end{proposition}

\begin{proof}
\noindent
\medskip
{\it Step 1.}
Multiplying third equation in~\eqref{LSS} by $\Psi$ and integrating on 
$\Omega$, after noticing that
$$\int_\Omega\nabla\cdot(\Psi(V+U_0))\Psi=0$$ because both $V$ and $U_0$ are divergence-free, we get 
\begin{equation}
\label{ener-Psi}
\int_\Omega |\nabla \Psi|^2+\frac{\lambda}{2}\int_\Omega \Psi^2
+\lambda\int_\Omega \nabla\cdot[\Theta_0(V+U_0)]\Psi=0.
\end{equation}
The latter integral on the left-hand side
can be estimated by
\begin{equation*}
\frac12\int_\Omega|\nabla \Psi|^2
+\|\Theta_0\|_\infty^2\int_\Omega |V|^2+\int_\Omega|\Theta_0U_0|^2.
\end{equation*}
The latter integral is finite because, by \eqref{uoto}, we have got 
$|\Theta_0 U_0|\in L^2(\R^3)$.
Thus, as $0\le \lambda\le1$, we get the estimate
\begin{equation}
\label{est:psi0}
\frac12\int_\Omega|\nabla\Psi|^2+\frac{\lambda}{2}\int_\Omega \Psi^2
\le \lambda\Bigl(\|\Theta_0\|_\infty^2\int_\Omega|V|^2+\int_\Omega |\Theta_0U_0|^2\Bigr).
\end{equation}
By the Poincaré inequality and by estimate \eqref{est:psi0}, we deduce that,
in order to establish~\eqref{APB}, it is sufficient to prove that
\begin{equation}
\label{APB2}
\int_\Omega |\nabla V|^2
\le C_1
\end{equation}
for some $C_1=C_1(\Omega,F,\rho,U_0,\Theta_0)>0$ independent on $\lambda\in[0,1]$.

Multiplying the first equation of system \eqref{LSS} by $V$, after some integration by parts,
we get
\begin{equation}
\label{ener-V}
\begin{split}
\int_\Omega |\nabla V|^2+\frac{\lambda}{2}\int_{\Omega}|V|^2
=&\lambda\Bigl[
\int_\Omega F_0\cdot V-\int_\Omega(V\cdot\nabla U_0)\cdot V\\
&+\int_\Omega\left((\Psi+\Theta_0)\rho\nabla |\cdot|^{-1}\right)\cdot V
{\color{blue}+\langle F, V\rangle}
\Bigr].
\end{split}
\end{equation}
{\color{blue}Here $\langle F, V\rangle$ is the duality product $H^{-1}(\Omega)$-$H^1_0(\Omega)$.} 
In the above identity, the convergence of the integral 
$\int_\Omega\left((\Psi+\Theta_0)\rho\nabla |\cdot|^{-1}\right)\cdot V$ deserves an explanation: the fact that $\Psi\rho\nabla|\cdot|^{-1}\cdot V$ and
$\Theta_0\rho\nabla|\cdot|^{-1}\cdot V$ are both integrable on $\Omega$
follows from estimates
\eqref{hardy-arg}--\eqref{HL:est} below.
The other integrals in~\eqref{ener-V} are obviously convergent because of our assumptions on $(U_0,\Theta_0)$. 

But rather than working directly with the energy inequality~\eqref{ener-V}, we proceed 
by contradiction: assume that there exist a sequence 
$(\lambda_k)\subset[0,1]$ and a sequence of solutions $(V_k,\Psi_k)\subset {\bf H}(\Omega)$ to problem~\eqref{LSS}-\eqref{LSS:B}, with $\lambda_k$ instead of $\lambda$, such that
\[
J_k:=\Big(\int_\Omega |\nabla V_k|^2\Bigr)^{1/2}\to+\infty.
\]
We also set 
\[
L_k:=\Big(\int_\Omega |\nabla \Psi_k|^2\Bigr)^{1/2}
\]
and we introduce the normalized functions
\[
\widehat V_k=\frac{V_k}{J_k}
\quad\text{and}\quad
\widehat\Psi_k=\frac{\Psi_k}{L_k}, 
\]
so that  $(\widehat V_k,\widehat \Psi_k)$ is a bounded sequence in ${\bf H}(\Omega)$.
After extracting a suitable subsequence, we can assume that
 $(\widehat V_k,\widehat \Psi_k)\to (\widetilde V,\widetilde \Psi)$
weakly in ${\bf H}(\Omega)$ and strongly in ${\bf L}^p(\Omega)$, for $p\in [2,6)$.
We can also assume that $\lambda_k\to\lambda_0$, for some $\lambda_0\in[0,1]$.

{\it Step 2. Excluding the case: $\limsup_{k\to+\infty}J_k/L_k<\infty$.} 

If by contradiction, $\limsup_{k\to+\infty}J_k/L_k<\infty$, then
after a new extraction of a subsequence, we can assume that
there exists $\gamma\ge0$ such that
\[
J_k/L_k\to\gamma.
\]
In fact, $\gamma>0$ by estimate~\eqref{est:psi0}.
Moreover, 
as $J_k\to+\infty$, we must have $L_k\to+\infty$. 
Equation~\eqref{ener-Psi} holds with $(V_k,\Psi_k)$ instead of~$(V,\Psi)$, namely
\[
\label{ener-Psik}
\int_\Omega |\nabla \Psi_k|^2+\frac{\lambda_k}{2}\int_\Omega |\Psi_k|^2
+\lambda_k\int_\Omega \nabla\cdot[\Theta_0(V_k+U_0)]\Psi_k=0.
\]
After dividing term-by-term by $L_k^2$ and taking the limit as $k\to+\infty$, 
using that $\widehat \Psi_k\to\widetilde \Psi$ weakly in $H^1_0(\Omega)$, $\widehat V_k\to\widetilde V$ strongly in $L^2(\Omega)^3$ and $\Theta_0\in L^\infty(\Omega)$ as well as the fact that $\Theta_0U_0\in L^2(\R^3)$, implies
\[
\frac{1}{L_k^2}\bigl|\int_\Omega \nabla\cdot (\Theta_0U_0)\Psi_k\bigr|
\le \frac{C}{L_k}\|\Theta_0U_0\|_{L^2(\Omega)}\to0.
\]
Hence, we get the equation 
\begin{equation}
\label{exl}
1+\frac{\lambda_0}{2}\int_\Omega |\widetilde \Psi|^2
=-\lambda_0\gamma\int_\Omega\nabla\cdot(\Theta_0\widetilde V)\,\widetilde\Psi.
\end{equation}
The weak formulation of the third equation of~\eqref{LSS}, written for $(V_k,\Psi_k)$ 
and $\lambda_k$,
gives, for all $\chi\in C^\infty_0 (\Omega)$,
\[
\int_\Omega \nabla\Psi_k \cdot\nabla \chi=
\lambda_k\Bigl[\int_\Omega[\Psi_k+x\cdot\nabla\Psi_k]\chi
-\int_\Omega\nabla\cdot[(\Psi_k+\Theta_0)(V_k+U_0)]\chi\Bigr].
\]
Let us divide this identity by $L_k^2$ and let $k\to+\infty$. All the terms $\frac{1}{L_k^2}\int_\Omega$ which are  linear with respect to  $\{V_k,\Psi_k\}$ tend to zero. Hence, in the limit,  we find
\[
\lambda_0\gamma\int_\Omega \nabla\cdot(\widetilde\Psi\widetilde V) \chi=0
\qquad
\text{for all $\chi\in C^\infty_0 (\Omega)$}.
\]
But $\lambda_0\gamma\not=0$ by equation~\eqref{exl}, so
\[
\int_\Omega \nabla\cdot (\widetilde\Psi\widetilde V)\chi=0
\qquad
\text{for all $\chi\in C^\infty_0 (\Omega)$}.
\]
This in turn implies that
\[
\widetilde V\cdot\nabla \widetilde\Psi=0.
\]
But then,
\[
\int_\Omega\nabla\cdot(\Theta_0\widetilde V)\widetilde\Psi
=-\int_\Omega\Theta_0\widetilde V\cdot\nabla\Psi=0
\]
which contradicts equation~\eqref{exl}.
This excludes that $\limsup_{k\to+\infty}J_k/L_k<\infty$.

{\it Step 3.} 
We reduced ourselves to the case $\limsup_{k\to+\infty} J_k/L_k=+\infty$.
After extracting a new subsequence, we can assume that $L_k/J_k\to0$.
Equation~$\eqref{ener-V}$ holds true for $(V_k,\Psi_k)$ and $\lambda_k$  instead of $(V,\Psi)$ and $\lambda$.
Let us divide it by $J_k^2$ and study the limit of each term, as $k\to+\infty$.
We have
\[
\frac{1}{J_k^2}\int_\Omega |\nabla V_k|^2=1,
\qquad
\frac{1}{J_k^2}\int_\Omega|V_k|^2\to\int_\Omega|\widetilde V|^2,
\qquad
\frac{1}{J_k^2}\int_\Omega F_0\cdot  V_k\to0,
\]
and
\[
{\color{blue}\frac{1}{J_k^2}\langle F, V_k\rangle \to0},
\qquad
\frac{1}{J_k^2}\int_\Omega(V_k\cdot\nabla U_0)\cdot  V_k
\to
\int_\Omega(\widetilde V\cdot\nabla U_0)\cdot \widetilde V,
\]
because  $|\nabla U_0|\in L^2(\Omega)$ and $\widehat V_k\to\widetilde V$ strongly in $L^p(\Omega)^3$ for $p\in [2,6)$.
For the next term, we rely on the Hardy inequality: as $\Psi_k$ and $V_k$ belong to $H^1_0(\Omega)$ we can write


\begin{equation}
\label{hardy-arg}
\begin{split}
\frac{1}{J_k^2}
\Bigl|\int_\Omega \Psi_k\rho\nabla\bigl({|\cdot|^{-1}}\bigr)\cdot V_k\Bigr|
&\le 
 \frac{C}{J_k^2}\,\Bigl\| \frac{\Psi_k}{|\cdot|}\Bigr\|_{L^2(\Omega)}
 \Bigl\| \frac{V_k}{|\cdot|}\Bigr\|_{L^2(\Omega)}\\
&\le 
\frac{C}{J_k^2}\|\nabla\Psi_k\|_{L^2(\Omega)}\|\nabla V_k\|_{L^2(\Omega)}\\
&\le \frac{C}{J_k}\|\nabla \Psi_k\|_{L^2(\Omega)}
=C\frac{L_k}{J_k}\to0.
\end{split}
\end{equation}
The function $\Theta_0$ does not belong to $H^1_0(\Omega)$.
However, it follows from~\eqref{uoto} that $\Theta_0/|\cdot|\in L^2(\R^3)$.
Therefore, the term of \eqref{ener-V} containing $\Theta_0$
can be estimated as in~\eqref{hardy-arg}.
Namely,
\begin{equation}\label{HL:est}
\begin{split}
\frac{1}{J_k^2}
\Bigl|\int_\Omega \Theta_0\rho\nabla\bigl({|\cdot|^{-1}}\bigr)\cdot V_k\Bigr|
&\le 
 \frac{C}{J_k^2}\,\Bigl\| \frac{\Theta_0}{|\cdot|}\Bigr\|_{L^2(\Omega)}
 \Bigl\| \frac{V_k}{|\cdot|}\Bigr\|_{L^2(\Omega)}\\
&\le 
\frac{C}{J_k^2}\|\nabla V_k\|_{L^{2}(\Omega)}
= \frac{C}{J_k}\to0.
\end{split}
\end{equation}
%
The above calculations lead to the identity
\begin{equation}
\label{idela}
1+\frac{\lambda_0}{2}\int_\Omega |\widetilde V|^2
=-\lambda_0\int_\Omega(\widetilde V\cdot\nabla U_0)\cdot\widetilde V,
\end{equation}
which implies  $\lambda_0\not=0$.
So, for large enough $k$, we have $\lambda_k\not=0$
and we can normalize the pressure putting
\[
\widehat{P_k}:=\frac{P_k}{\lambda_k J_k^2}.
\]
We now go back to  first equation in~\eqref{LSS}, written for $(V_k,\Psi_k)$ and  
$\lambda_k$ instead of $(V,\Psi)$ and $\lambda$.
Dividing by $\lambda_kJ_k^2$ we obtain
\[\label{Hardy}
\begin{split}
\widehat V_k\cdot\nabla\widehat V_k+\nabla\widehat P_k
=&\frac{1}{J_k}\Biggl(\frac{\Delta\widehat V_k}{\lambda_k}+\widehat V_k+x\cdot \nabla \widehat V_k+\frac{F_0}{J_k}\\
&\qquad-U_0\cdot\nabla\widehat V_k-\widehat V_k\cdot\nabla U_0+\Bigl(\frac{L_k}{J_k}\widehat\Psi_k+\frac{\Theta_0}{J_k}\Bigr)\rho\nabla(|\cdot|^{-1})+\frac{F}{J_k}\Biggr).
\end{split}
\]
More precisely, we consider the weak formulation of the above equation: testing with an arbitrary solenoidal vector field $\eta\in C^\infty_{0,\sigma}(\R^3)^3$, after integrating on 
$\Omega$ and letting $k\to+\infty$, the terms obtained in the right-hand side 
vanish, in the limit.
{\color{blue}
Indeed, for the two last terms, 
we obviously have
\[
\frac{\langle F,\eta\rangle}{J_k^2}\to0
\]
and, following the calculations in \eqref{HL:est}, 
\[
\frac{1}{J_k^2}\int_\Omega  \Theta_0\rho\nabla(|\cdot|^{-1})\cdot\eta\to0.
\]
}
Moreover,  we also have
\[
\frac{L_k}{J_k^2}\int\widehat\Psi_k\rho\nabla(|\cdot|^{-1})\cdot \eta\to0
\]
because $L_k/J_k^2\to0$, $\rho$ is a bounded function,
and $|\int_\Omega \widehat\Psi_k\nabla(|\cdot|^{-1})\cdot \eta|$ can be bounded uniformly with respect to~$k$ applying the Hardy inequality as in~\eqref{hardy-arg}.
The terms obtained testing with $\eta$ the other terms on the right-hand side also vanish (because $\widehat V_k$ is bounded in $H^1_0(\Omega)^3$ and $|U_0|$ and $|\nabla U_0|$ are both in $L^\infty(\Omega)$, and because $J_k\to+\infty)$.
But $\int_\Omega\nabla \widehat P_k\cdot \eta=0$, therefore, we find in the limit
\begin{equation}
 \label{seuler}
 \int_\Omega (\widetilde V\cdot\nabla \widetilde V)\cdot\eta=0,
 \qquad\text{for all $\eta\in C^\infty_{0,\sigma}(\R^3)^3$.}
\end{equation}
This means that $\widetilde V\in H^1_0(\Omega)$ is a stationary solution of the Euler equations.
At this stage, the proof can be finished exactly as in the paper by Korobkov and Tsai~\cite{KT16}:
there exists $\widetilde P\in L^3(\Omega)$, such that $\|\nabla \widetilde P\|_{L^{3/2}(\Omega)}<\infty$, satisfying
\[
\begin{cases}
\widetilde V\cdot\nabla \widetilde V=-\nabla \widetilde P &\text{in $\Omega$}\\
\nabla\cdot \widetilde V=0 &\text{in $\Omega$}\\
\widetilde V=0& \text{on $\partial \Omega$}.
\end{cases} 
\]
Then, going back to~\eqref{idela}, and using once more that $U_0,\nabla U_0$ are in $L^\infty(\R^3)$, we find
\[
\begin{split}
1+\frac{\lambda_0}{2}\int_\Omega|\widetilde V|^2
&=-\lambda_0\int_\Omega(\widetilde V\cdot\nabla U_0)\cdot\widetilde V
=\lambda_0\int_\Omega(\widetilde V\cdot \nabla \widetilde V)\cdot U_0\\
&=-\lambda_0\int_\Omega\nabla\widetilde P\cdot U_0
=-\lambda_0\int_\Omega\nabla\cdot(PU_0)\\
&=0.
\end{split}
\]
The last equality relies on a classical result on the stationary Euler equation~\cite{KP83}*{Lemma 4} (see also~\cite{AM84}*{Theorem~2.2}), implying that in addition to the above properties, the pressure $P$ can be taken additionally such that $P(x)\equiv0$ a.e. on $\partial\Omega$, with respect to the two-dimensional Hausdorff measure.
From the last equality we get a contradiction.
\end{proof}

\subsection{Existence of solutions to the perturbed elliptic system in bounded domains}
Let $\Omega$ be a bounded domain with a smooth boundary.
As before, we take $\rho\in C_b(\R^3)$, but here we additionally
assume that the support of $\rho$ does not contain the origin, in a such way that 
$\rho\nabla(|\cdot|^{-1})\in L^2(\R^n)\cap L^\infty(\R^n)$.

Let us define the linear map $L_\rho$ and the nonlinear map $N$, 
\[
\begin{split}
L_\rho(V,\Psi)\!:=&
\!\Bigl(V+x\cdot \nabla V-U_0\cdot\nabla V-V\cdot\nabla U_0+\Psi\rho\nabla(\textstyle{|\cdot|^{-1}}), \\
& \qquad \qquad  \qquad \qquad  \qquad \qquad  \qquad \qquad \Psi+x\cdot\nabla\Psi-\nabla\cdot\bigl(\Psi U_0 + \Theta_0V\bigr)
\Bigr)\\
N(V,\Psi)
\!:=&\!\Bigl(-U_0\cdot\nabla U_0-V\cdot\nabla V+\Theta_0\rho\nabla(|\cdot|^{-1})\, ,
      -\nabla\cdot(\Psi V+\Theta_0 U_0)\Bigr).\\
\end{split}
\]
We also introduce the following nonlinear map
\begin{equation}
\begin{split}
G_\rho(V,\Psi):=
L_\rho(V,\Psi)+N(V,\Psi).
 \end{split}
\end{equation}
In this way, our system~\eqref{LSS}, in the case $\lambda=1$, can be rewritten as
\begin{equation}
\label{eq:G}
(-\Delta V+\nabla P-F,-\Delta\Psi)=G_\rho(V,\Psi).
\end{equation}

We endow the dual space ${\bf H}(\Omega)'$ with the usual norm of dual Banach spaces.

\begin{lemma}
\label{lem:com}
Let $\Omega$ be a bounded domain with a smooth boundary and
$\rho\in C_b(\R^3)$, such that $0\not\in {\rm supp}(\rho)$.
The nonlinear map $G_\rho$ is well defined as a map
$G_\rho\colon {\bf H}(\Omega)\to {\bf L}^{3/2}(\Omega)$
and is compact as a map
$G_\rho\colon {\bf H}(\Omega)\to {\bf H}(\Omega)'$.
 \qquad
\end{lemma}

\begin{proof}
Notice that $G_\rho\colon {\bf H}(\Omega)\to {\bf L}^{3/2}(\Omega)$  is well defined.
Indeed, if
$(V,\Psi)\in {\bf H}(\Omega)\subset {\bf L}^6(\Omega)$,  
then the components of $V\cdot \nabla V$ and $\nabla \cdot (\Psi V)$ do belong to $L^{3/2}(\Omega)$.
In the same way, using the conditions~\eqref{uoto} on $U_0$ and $\Theta_0$, one easily checks that all the other terms defining $G_\rho(V,\Psi)$ belong also to $L^{3/2}(\Omega)$ (or even to a smaller space). The presence of the function $\rho$ cutting out the singularity of $\nabla(|\cdot|^{-1})$ near the origin is important here.

As the Sobolev embedding ${\bf H}(\Omega)\subset {\bf L}^6(\Omega)$ is continuous,
by the previous considerations
the map $G_\rho\colon {\bf H}(\Omega)\to {\bf L}^{3/2}(\Omega)$ is continuous.
Moreover, every function $f\in {\bf L}^{3/2}(\Omega)$ can be identified to an element of ${\bf H}(\Omega)'$
through the usual duality $h\mapsto \int_{\Omega} f\cdot h$, where $h\in {\bf H}(\Omega)$.
Adopting this identification, it results that the map 
$G_\rho\colon {\bf H}(\Omega)\to {\bf H}(\Omega)'$ is continuous.

Let us prove that, in fact, $G_\rho\colon {\bf H}(\Omega)\to {\bf H}'(\Omega)$ is
a compact operator.
For every $(V,\Psi), (\widetilde V,\widetilde \Psi)\in {\bf H}(\Omega)$, 
denoting
\[
v:=\widetilde V-V,\quad\text{and}\quad \psi:=\widetilde\Psi-\Psi,
\]
we have,
\[
\begin{split}
N(\widetilde V,\widetilde \Psi)-N(V,\Psi)
=\Bigl(-(v+V)\cdot\nabla v-v\cdot\nabla V, -\nabla\cdot((\psi+\Psi)v+\psi V)\Bigr).
\end{split}
\]
Now, if $(V_k,\Psi_k)$ is a bounded sequence in ${\bf H}(\Omega)$, 
with $\|(V_k,\Psi_k)\|_{{\bf H}(\Omega)}\le R$,
then there exists $(\widetilde V,\widetilde\Psi)\in {\bf H}(\Omega)$ such that, after extraction of a subsequence,  
$$
\text{
$(v_k,\psi_k):=(\widetilde V-V_k,\widetilde \Psi-\Psi_k)\to0$
weakly in ${\bf H}(\Omega)$ and strongly in ${\bf L}^3(\Omega)$.}
$$
For any $\Phi\in {\bf H}(\Omega)$, we have, after some integration by parts,
\[
\Bigl|
\int_{\Omega} N(\widetilde V,\widetilde\Psi)\cdot \Phi-\int_{\Omega}N(V_k,\Psi_k)\cdot \Phi\Bigr|
\le  CR\bigl(\|v_k\|_3+\|\psi_k\|_3\bigr) \|\Phi\|_{{\bf H}(\Omega)}
\]
for some constant $C>0$ independent on~$k$ and $\Phi$.
It follows that
\[
\begin{split}
\|N(\widetilde V,\widetilde \Psi)-N(V_k,\Psi_k)\|_{{\bf H}'(\Omega)}
&:=\sup_{\|\Phi\|_{{\bf H}(\Omega)}=1} \,
 \Bigl|
\int_{\Omega} [N(\widetilde V,\widetilde\Psi)-N(V_k,\Psi_k)]\cdot \Phi\Bigr|\\
&\le CR \bigl(\|v_k\|_3+\|\psi_k\|_3\bigr)\to0.
\end{split}
\]
For the linear terms we have also,
\[
\|L_\rho(V_k,\Psi_k)-L_\rho(\widetilde V,\widetilde\Psi)\|_{{\bf H}(\Omega)'}
=\| L_\rho(v_k,\psi_k)\|_{{\bf H}(\Omega)'}\to0,
\]
as one can check using conditions~\eqref{uoto}.
Here the condition $0\not\in{\rm supp}(\rho)$ is useful to prove that
$\|\psi_k\,\rho\nabla(|\cdot|^{-1})\|_{H^1(\Omega)'}\to0$,
which is part of the previous claim;
no other difficulty arises for the other terms of $L_\rho(v_k,\psi_k)$.

Hence,
\[
G_\rho(V_k,\Psi_k) \to G_\rho(\widetilde V,\widetilde\Psi)
\qquad\text{strongly in ${\bf H}(\Omega)'$.}
\]
\end{proof}

\begin{proposition}
\label{prop:exi-rho}
Let $\Omega$ be a bounded domain with a smooth boundary and $\rho\in C^\infty_b(\R^3)$,
such that $0\not\in{\rm supp}(\rho)$.
Let $(U_0,\Theta_0)$ be as in~\eqref{uoto} and assume that {\color{blue} $F\in H^{-1}(\Omega)^3$}.
Then the system
\begin{equation}
\label{SSr}
\begin{split}
-\Delta V+\nabla P
&=V+x\cdot\nabla V+F_0+F_1(V)+(\Psi+\Theta_0)\,\rho\nabla(|\cdot|^{-1})+F\\
\nabla\cdot V&=0,\\
-\Delta\Psi
&=\Psi+x\cdot\nabla\Psi-\nabla\cdot\bigl((\Psi+\Theta_0)(V+U_0)\bigr)
\end{split}
\qquad x\in\Omega,
\end{equation}
supplemented with the Dirichlet boundary conditions
\begin{equation}\label{SSr:D}
    V=0, \qquad \Psi=0\qquad \text{on} \qquad \partial\Omega,
\end{equation}
has a solution $(V,\Psi) \in {\bf H}(\Omega)$.
\end{proposition}

\begin{proof}

Let $T\colon {\bf H}(\Omega)'\to {\bf H}(\Omega)$ be the isomorphism given by the Riesz representation theorem for Hilbert spaces, where the Hilbert space ${\bf H}(\Omega)$
is endowed with the scalar product
\[
\bigl((V,\Psi),(V',\Psi')\bigr)\mapsto \int_{\Omega}\nabla V\cdot\nabla V'
+\int_\Omega \nabla\Psi\cdot\nabla\Psi'.
\]
{\color{blue} By assumption on $F$, we have $(F,0)\in {\bf H}(\Omega)'$ (the fourth component is zero because we considered no forcing term in the equation of the temperature).}
The weak formulation of equation~\eqref{eq:G} reads
\begin{equation}
\label{WCT}
(V,\Psi)=T(G_\rho(V,\Psi))+T((F,0)).
\end{equation}

By Lemma~\ref{lem:com}, 
the nonlinear map $T\circ G_\rho\colon {\bf H}(\Omega)\to {\bf H}(\Omega)$ is compact. Hence, the map  $(V,\Psi)\mapsto T(G_\rho(V,\Psi))+T((F,0))$ is compact on ${\bf H}(\Omega)$.

For every $\lambda\in[0,1]$, if $(V,\Psi)=\lambda (T\circ G_{\rho})(V,\Psi)+\lambda T((F,0))$,
{\it i.e.}, if $(V,\Psi)$ is a solution of~\eqref{LSS}, 
then $\|(V,\Psi)\|_{{\bf H}(\Omega)}\le C_0$, where $C_0$ is the constant, independent
on $\lambda$, obtained in Proposition~\ref{prop:1ae}.
The Leray-Schauder fixed-point theorem (see~{\it e.g.}~\cite{LR16}*{p.529})
then implies that the map $(V,\Psi)\mapsto T(G_\rho(V,\Psi))+T((F,0))$ has a fixed point 
$(V_\rho,\Psi_\rho)\in {\bf H}(\Omega)$, such that 
$\|(V_\rho,\Psi_\rho)\|_{{\bf H}(\Omega)}\le C_0$.
\end{proof}

\subsection{Existence of solutions of the perturbed elliptic system in the whole space}

In this subsection we choose, once and for all, a cut-off function
$\rho\in C^\infty_b(\R^3)$ such that, $\rho(x)=0$ if $|x|\le 1/2$, $0\le \rho(x)\le1$ if
$1/2\le |x|\le 1$ and  $\rho(x)=1$ if $|x|\ge1$.
Then we set, for $x\in\R^3$,
\[
\rho_k(x):=\rho(kx),
\]
so that $0\le \rho_k(x)\le 1$ and $\rho_k\to1$ a.e. in $\R^3$ as $k\to+\infty$.

\begin{proposition}
\label{prop:invading}
Let $(U_0,\Theta_0)$ be as in~\eqref{uoto} and {\color{blue}$F\in H^{-1}(\R^n)^3$}.
Let $k\in\N$, $k\ge1$ and $B_k$ the open
ball of $\R^3$ centered at the origin and of radius~$k$.
Let $(V_k,\Psi_k)\in {\bf H}(B_k)$ be a solution
of problem~\eqref{SSr}-\eqref{SSr:D} with $\Omega=B_k$ and $\rho=\rho_k$.
Then there exists a constant $C_1=C_1(F,U_0,\Theta_0)>0$, independent on~$k$,
such that
\begin{equation}
\label{APBi}
\int_{B_k}\Bigl(|V_k|^2+\Psi_k^2+|\nabla V_k|^2+|\nabla \Psi_k|^2\Bigr)
\le C_1. 
\end{equation}
\end{proposition}

\begin{proof}
The proof has a similar structure to that of Proposition~\ref{prop:1ae}.

{\it Step 1.}
First of all, by estimates~\eqref{uoto} and \eqref{est:psi0}, in the case $\Omega=B_k$ and $\lambda=1$, we
have 
\begin{equation}
\label{est:psi}
\int_{B_k}|\nabla\Psi_k|^2+\int_{B_k}\Psi_k^2
\le 2\Bigl(\|\Theta_0\|_{L^\infty(\R^3)}^2\int_{B_k}|V_k|^2+\int_{\R^3} |\Theta_0U_0|^2\Bigr).
\end{equation}
Hence, it is sufficient to prove that
\begin{equation}
\label{APBv}
\int_{B_k}\bigl(\frac12|V_k|^2+|\nabla V_k|^2\bigr) \le C_1. 
\end{equation}
With a slight change of notations with respect to  Proposition~\ref{prop:1ae}, we now set
\[
J_k:=\Bigl(\int_{B_k}\bigl(\frac12|V_k|^2+|\nabla V_k|^2\bigr)\Bigr)^{1/2},
\quad\text{and}\quad
L_k:=\Bigl(\int_{B_k}\bigl(\frac12 \Psi_k^2+|\nabla \Psi_k|^2\bigr)\Bigr)^{1/2}
\]
and
\[
\widehat V_k:=\frac{V_k}{J_k},
\quad
\widehat P_k=\frac{P_k}{J_k^2},
\quad\text{and}\quad
\widehat\Psi_k:=\frac{\Psi_k}{L_k}.
\]
Let us assume, by contradiction, that \eqref{APBv} does not hold. 
Thus, there exists a subsequence of solutions 
$(V_k,\Psi_k)\in {\bf H}(B_k)$ 
of problem~\eqref{SSr}-\eqref{SSr:D} with $\Omega=B_k$ and $\rho=\rho_k$
such that 
\[
J_k\to+\infty.
\]
The boundedness of the sequence $(\widehat V_k,\widehat \Psi_k)$ in ${\bf H}(B_k)$  (or, more precisely, of the sequence obtained extending $(\widehat V_k,\widehat \Psi_k)$ to the whole $\R^3$ via the classical extension theorem for Sobolev spaces), implies that there exists 
$(\widetilde V,\widetilde \Psi)\in {\bf H}(\R^3)$, such that, after extraction of a subsequence, 
\[
(\widehat V_k,\widehat \Psi_k)\to (\widetilde V,\widetilde\Psi)
\]
weakly in ${\bf H}(\R^3)$ and strongly in ${\bf L}^p_{\rm loc}(\R^3)$, for $2\le p<6$. 
The divergence-free condition for $\widetilde V$ follows from the fact that,
for every test function $\varphi\in C^\infty_0(\R^3)$, one has
$\int  \widetilde V\cdot \nabla \varphi=\lim_k\int \widetilde V_k\cdot \nabla \varphi=0$.

{\it Step 2. Excluding that  $\limsup_{k\to+\infty}J_k/L_k<\infty$.}
Assume for the moment that $\limsup_{k\to+\infty}J_k/L_k<\infty$.
So we must have also $L_k\to+\infty$.
Then, after extraction, we have $J_k/L_k\to \gamma$, for some real $\gamma>0$,
because of inequality~\eqref{est:psi}.
But the following identity holds true, just like~\eqref{ener-Psi},
\begin{equation}
\label{ener-Psik2}
\int_{B_k} |\nabla \Psi_k|^2+\frac{1}{2}\int_{B_k} \Psi_k^2
+\int_{B_k} \nabla\cdot[\Theta_0(V_k+U_0)]\Psi_k=0.
\end{equation}
Let us divide it by $L_k^2$ and take $k\to+\infty$.
We claim that
\[
\frac{1}{L_k^2}\int_{B_k}\nabla\cdot(\Theta_0 V_k)\Psi_k
=\frac{1}{L_k^2}\int_{B_k}\nabla\Theta_0\cdot V_k\Psi_k\to\gamma\int_{\R^3}(\nabla\Theta_0\cdot \widetilde V)\widetilde\Psi.
\]
Indeed, let $\Omega$ be a bounded domain and $k$ large enough so that $\Omega\subset B_k$. 
We make use of the fact that condition~\eqref{uoto} 
implies $\nabla\Theta_0\in L^4(\R^3)$, and that $H^1(\R^3)$ is continuously embedded
in $L^{8/3}(\R^3)$. So we have
\begin{equation}
\label{calula}
\begin{split}
\Bigl|\frac{1}{L_k^2}\int_{B_k}(\nabla\Theta_0\cdot & V_k)\Psi_k-\gamma\int_{\R^3}(\nabla\Theta_0\cdot \widetilde V)\widetilde\Psi\Bigr|\\
&\le
\Bigl|\frac{J_k}{L_k}\int_{B_k}(\nabla\Theta_0\cdot \widehat V_k)\widehat\Psi_k-\gamma\int_{B_k}(\nabla\Theta_0\cdot \widetilde V)\widetilde\Psi\Bigr| \\
&\qquad \qquad +\gamma\|\nabla\Theta_0\|_{L^4(B_k^c)}
 \|\widetilde V\|_{L^{8/3}(\R^3)}\|\widetilde \Psi\|_{L^{8/3}(\R^3)}\\
&\le
\Bigl|\int_{\Omega}\Bigl(\frac{J_k}{L_k}(\nabla\Theta_0\cdot \widehat V_k)\widehat\Psi_k-\gamma(\nabla\Theta_0\cdot \widetilde V)\widetilde\Psi\Bigr)\Bigr|
+C\gamma\|\nabla\Theta_0\|_{L^4(\Omega^c)},
\end{split}
\end{equation}
because $\widehat V_k$ and $\widehat \Psi_k$ can be extended
to $\R^3$ and
that such extensions are bounded in $H^1(\R^3)$, and so in $L^{8/3}(\R^3)$, 
uniformly with respect to $k$.
The first term in the right-hand side 
tends to zero as $k\to+\infty$ by the compact embedding of ${\bf H}(\Omega)$ 
into ${\bf L}^{8/3}(\Omega)$. The second term 
can be taken as small as we wish, taking
$\Omega=B_R$, $k>R$, choosing a radius $R>0$ large enough.

Hence, we get from~\eqref{ener-Psik2}
\begin{equation}
\label{gno}
1=\gamma\int_{\R^3}\nabla\cdot(\Theta_0\widetilde V)\widetilde\Psi.
\end{equation}

Let us consider an arbitrary test function $\chi\in C^\infty_c(\R^3)$ and $k$ large enough so that the support of $\chi$
is contained in $B_k$.
From the second equation of~\eqref{SSr}, written for $(V_k,\Psi_k)$, 
we obtain
\[
\int_{B_k} \nabla\Psi_k \cdot\nabla \chi=
\int_{B_k}[\Psi_k+x\cdot\nabla\Psi_k]\chi
-\int_{B_k}\nabla\cdot[(\Psi_k+\Theta_0)(V_k+U_0)]\chi.
\]
Dividing by $L_k^2$ and letting $k\to+\infty$, all the integrals $\frac{1}{L_k^2}\int_{B_k}\ldots$
with linear terms in $V_k$ and $\Psi_k$ in the above identity go to zero.
On the other hand, proceedings as in~\eqref{calula},
\[
\frac{1}{L_k^2}\int_{B_k} \nabla\cdot(\Psi_k V_k)\chi
\to \gamma\int_{\R^3}\nabla\cdot(\widetilde \Psi \widetilde V)\chi.
\]
Then it follows that
\[
\gamma\int_{\R^3} \nabla\cdot(\widetilde\Psi\widetilde V) \chi=0
,\qquad
\text{for all $\chi\in \mathcal{D}(\R^3)$}.
\]
But $\gamma\not=0$ by~\eqref{gno}, so
\[
\int_{\R^3} \nabla\cdot (\widetilde\Psi\widetilde V)\chi=0
,\qquad
\text{for all $\chi\in \mathcal{D}(\R^3)$}.
\]
This in turn implies that
$\widetilde V\cdot\nabla \widetilde\Psi=0$.
But then
\[
\int_{\R^3}\nabla\cdot(\Theta_0\widetilde V)\widetilde\Psi
=-\int_{\R^3}\Theta_0\widetilde V\cdot\nabla\Psi=0
\]
which contradicts~\eqref{gno}.
This excludes that $\limsup_{k\to+\infty}J_k/L_k<\infty$.

{\it Step 3. }
We reduced ourselves to the case $\limsup_{k\to+\infty}J_k/L_k=\infty$.
Multiplying the first equation in~\eqref{SSr} by $V_k$ 
and integrating on $B_k$ gives, in a similar way as we did in~\eqref{ener-V}
\begin{equation}
\label{jk2}
J_k^2=\int_{B_k}F_0V_k-\int_{B_k}(V_k\cdot\nabla U_0)\cdot V_k
+\int_{B_k}(\Psi_k+\Theta_0)\rho_k\nabla\bigl({|\cdot|^{-1}}\bigr)\cdot V_k+\langle F,V_k\rangle.
\end{equation}
Applying the Hardy inequality (\emph{cf.} \eqref{Hardy})  we get, as $k\to+\infty$,
\[
\begin{split}
\Bigl|
\frac{1}{J_k^2}\int_{B_k}(\Psi_k\rho_k\nabla(|\cdot|^{-1})\cdot V_k\Bigr|
&\le C\frac{L_k}{J_k}\Bigl\|\frac{\widehat\Psi_k}{|\cdot|}\Bigr\|_{L^2(B_k)} \, 
\Bigl\|\frac{\widehat V_k}{|\cdot|}\Bigr\|_{L^2(B_k)}\\
&\le C\frac{L_k}{J_k}\|\nabla \widehat\Psi_k\|_{L^2(B_k)} \|\nabla \widehat V_k\|_{L^2(B_k)}\to0.
\end{split}
\]
Moreover,
\[
\frac{1}{J_k^2}\int_{B_k}(V_k\cdot\nabla U_0)\cdot V_k
\to \int_{\R^3}(\widetilde V\cdot\nabla U_0)\cdot \widetilde V
\]
as one easily checks by reproducing the same calculations as in~\eqref{calula}, using that $\nabla U_0\in L^4(\R^3)$, by
condition~\eqref{uoto}. 
Therefore, dividing equation~\eqref{jk2} by $J_k^2$ and letting $k\to+\infty$
we find the identity
\begin{equation}
\label{absu}
    \int_{\R^3}(\widetilde V\cdot\nabla U_0)\cdot\widetilde V=-1.
\end{equation}

Dividing by $J_k^2$ the first equation in~\eqref{SSr} satisfied by $(V_k,\Psi_k)$, we get
\[
\begin{split}
&\widehat V_k\cdot\nabla \widehat V_k+\nabla \widehat P_k\\
&\qquad=\frac{1}{J_k}\Bigl(\Delta\widehat V_k+\widehat V_k+x\cdot\nabla\widehat V_k+\frac{F_0}{J_k}
  -U_0\cdot\nabla \widehat V_k-\widehat V_k\cdot \nabla U_0\\
&\qquad\qquad\qquad\qquad  +\frac{L_k}{J_k}\widehat\Psi_k\rho_k\nabla(|\cdot|^{-1})
  +\frac{1}{J_k}\left(\Theta_0\rho_k\nabla(|\cdot|^{-1})+F\right)
  \Bigr).
\end{split}
\]
Writing the weak formulation of the above equation, {\it i.e.}, testing with an arbitrary $\eta\in C^\infty_{0,\sigma}(\R^3)$ and using that
$J_k\to+\infty$ and that $L_k/J_k$ remains bounded as $k\to+\infty$,
we deduce, in the limit:
\[
\int_{\R^3} (\widetilde V\cdot\nabla \widetilde V)\cdot \eta=0,
\qquad\text{for all $\eta\in C^\infty_{0,\sigma}(\R^3)^3$}.
\]
But $U_0\in L^4_\sigma(\R^3)$, $\widetilde V\in L^4(\R^3)$ and $\nabla\widetilde V\in L^2(\R^3)$. Approximating 
$U_0$ in the $L^4$-norm by test functions implies
\[
\int_{\R^3}(\widetilde V\cdot\nabla \widetilde V)\cdot U_0=0.
\]
This is in contradiction with~\eqref{absu}.
\end{proof}

\begin{theorem}
\label{prop:exi-R3}
Assume that $U_0,$ and $\Theta_0$ satisfy~\eqref{uoto}. 
Assume also that $F\in L^p(\R^3)^3$, for some $p\in [6/5,2]$.
Then elliptic system~\eqref{DSS} possess at least a solution $(V,\Psi)\in {\bf H}(\R^3)$.
\end{theorem}

\begin{proof}
Applying Proposition~\ref{prop:exi-rho} with $\Omega=B_k$ and $\rho=\rho_k$ ($k=1,2,\ldots)$, we get a sequence of solutions $(V_k,\Psi_k)\in {\bf H}(B_k)$.
By Proposition~\ref{prop:invading}, such a sequence is bounded in the ${\bf H}(B_k)$-norm
by a constant independent on~$k$. This implies that there exist $(V,\Psi)\in{\bf H}(\R^3)$ and a subsequence, still denoted $(V_k,\Psi_k)$,  such that
$(V_k,\Psi_k)\to (V,U)$ weakly in ${\bf H}(\Omega)$, for any bounded domain 
$\Omega\subset \R^n$.

It remains to prove that $(V,\Psi)$ is a weak solution of the elliptic problem~\eqref{DSS}.
Thus, we have to pass to the limit in all the terms of the variational formulation of problem~\eqref{SSr}-\eqref{SSr:D}
considered on the balls $B_k$. 

To this purpose, we make use of the compact embedding 
$H^1_{\rm loc}(\R^3)\subset L^p_{\rm loc}(\R^3)$, valid for any with $p\in [2,6)$.
If $\chi\in C^\infty_0(\R^3)$ is a test function then we get,
after a new extraction, $\Psi_k\to\Psi$ in the Lorentz space~$L^{3,1}(\R^3)$.
But $\rho_k\nearrow 1$ a.e. in $\R^3$ and so $\rho_k\nabla(|\cdot|^{-1})\to\nabla(|\cdot|^{-1})$ in $L^{3/2,\infty}(\R^3)$.
Hence,
\[
\int_{\R^3} \Psi_k\rho_k\nabla(|\cdot|^{-1})\chi\to\int_{\R^3}\Psi\nabla(|\cdot|^{-1})\chi. 
\] 
Similar considerations 
prove that all the other terms also pass to the limit. This gives the result.
\end{proof}

\section{Conclusion}

This section is devoted to deduce the result
of Theorem~\ref{th:main} from  Theorem~\ref{prop:exi-R3}.

\begin{proof}[Proof of Theorem~\ref{th:main}.]
Let $(V,\Psi)$ be as in Theorem~\ref{prop:exi-R3} and set
\[
u(x,t):=\frac{1}{\sqrt{2t}}(U_0+V)\Bigl(\frac{x}{\sqrt{2t}}\Bigr),
\quad
\text{and}
\quad
\theta(x,t):=\frac{1}{\sqrt{2t}}(\Theta_0+\Psi)\Bigl(\frac{x}{\sqrt{2t}}\Bigr).
\]
We have $(U_0,\Theta_0)\in {\bf L}^{3,\infty}(\R^3)$ by Proposition \ref{prop:heat} and 
$(V,\Psi)\in {\bf H}(\R^3)\subset  {\bf L}^{3,\infty}(\R^3)$.
Then, from the scaling properties 
$(u,\theta)\in L^\infty(\R^+,{\bf L}^{3,\infty}(\R^3))$.

Let us now address the continuity with respect to~$t$  and we detail only the continuity property at $0$, that is important to give a sense to the
initial condition $(u_0,\theta_0)|_{t=0}=(u_0,\theta_0)$.
We have 
$$\frac{1}{\sqrt{2t}}(U_0,\Theta_0)\bigl(\frac{x}{\sqrt{2t}}\bigr)
=e^{t\Delta}(u_0,\Theta_0)\to (u_0,\theta_0) \quad\text{as} \quad t\to0+
$$
in the weak-* topology of ${\bf L}^{3,\infty}(\R^3)$. 
Let us check that
$\frac{1}{\sqrt{2t}}(V,\Psi)\bigl(\frac{x}{\sqrt{2t}}\bigr)\to0$ in the same topology.
Indeed, if $\varphi\in {\bf L}^{3/2,1}(\R^3)$, we can approximate it
in the ${\bf L}^{3/2,1}$-norm with functions 
$\varphi_\epsilon\in {\bf L}^2\cap{\bf L}^{3/2,1}(\R^3)$.
Then is is enough to observe that 
$$
|\int_{\R^3}\frac{1}{\sqrt{2t}}(V,\Psi)\bigl(\frac{\cdot}{\sqrt{2t}}\bigr)\cdot\varphi_\epsilon|\le Ct^{1/4}\to0
\quad \text{as} \quad t\to0+.
$$
Therefore, $(u,\theta)\in C_w(\R^+,{\bf L}^{3,\infty}(\R^3))$.

Equalities~\eqref{timeest} follow immediately from the fact that
$(V,\Psi)=(U-U_0,\Theta-\Theta_0)\in {\bf H}^1$ and from the scaling properties of
the $L^2$ and the $\dot H^1$-norms.
\end{proof}


\begin{thebibliography}{10}

\bibitem{AM84}
{\sc Amick, Ch.~J.}
\newblock Existence of solutions to the nonhomogeneous steady Navier–Stokes
equations.
\newblock
{\em Indiana Univ. Math. J. 33}, (1984), 817–830.

\bibitem{BP23}
{\sc Bradshaw, Z., and Phelps, P.}
\newblock Spatial decay of discretely self-similar solutions to the Navier-Stokes equations.
\newblock {\em  Pure Appl. Anal. 5}, 2 (2023),  377--407.


\bibitem{BT17}
{\sc Bradshaw, Z., and Tsai, T.-P.}
\newblock Forward discretely self-similar solutions of the {N}avier-{S}tokes
  equations {II}.
\newblock {\em Ann. Henri Poincar\'{e} 18}, 3 (2017), 1095--1119.

\bibitem{BT18}
{\sc Bradshaw, Z., and Tsai, T.-P.}
\newblock Discretely self-similar solutions to the {N}avier-{S}tokes equations
  with {B}esov space data.
\newblock {\em Arch. Ration. Mech. Anal. 229}, 1 (2018), 53--77.

\bibitem{BT19}
{\sc Bradshaw, Z., and Tsai, T.-P.}
\newblock Discretely self-similar solutions to the {N}avier-{S}tokes equations
  with data in {$L_{\rm loc}^2$} satisfying the local energy inequality.
\newblock {\em Analysis \& PDE 12}, 8 (2019), 1943--1962.

\bibitem{B09}
{\sc Brandolese, L.}
\newblock Fine properties of self-similar solutions of the {N}avier-{S}tokes
  equations.
\newblock {\em Arch. Ration. Mech. Anal. 192}, 3 (2009), 375--401.

\bibitem{BS12}
{\sc Brandolese, L., and Schonbek, M.~E.}
\newblock Large time decay and growth for solutions of a viscous {B}oussinesq
  system.
\newblock {\em Trans. Amer. Math. Soc. 364}, 10 (2012), 5057--5090.

\bibitem{CD80}
{\sc Cannon, J.~R., and DiBenedetto, E.}
\newblock The initial value problem for the {B}oussinesq equations with data in
  {$L\sp{p}$}.
\newblock In {\em Approximation methods for {N}avier-{S}tokes problems ({P}roc.
  {S}ympos., {U}niv. {P}aderborn, {P}aderborn, 1979)}, vol.~771 of {\em Lecture
  Notes in Math}. Springer, Berlin, 1980, pp.~pp 129--144.

\bibitem{C04}
{\sc Cannone, M.}
\newblock Harmonic analysis tools for solving the incompressible
  {N}avier-{S}tokes equations.
\newblock In {\em Handbook of mathematical fluid dynamics. {V}ol. {III}}.
  North-Holland, Amsterdam, 2004, pp.~161--244.

\bibitem{CMP94}
{\sc Cannone, M., Meyer, Y., and Planchon, F.}
\newblock Solutions auto-similaires des \'{e}quations de {N}avier-{S}tokes.
\newblock In {\em S\'{e}minaire sur les \'{E}quations aux {D}\'{e}riv\'{e}es
  {P}artielles, 1993--1994}. \'{E}cole Polytech., Palaiseau, 1994, pp.~Exp. No.
  VIII, 12.

\bibitem{CP96}
{\sc Cannone, M., and Planchon, F.}
\newblock Self-similar solutions for {N}avier-{S}tokes equations in {${\bf
  R}^3$}.
\newblock {\em Comm. Partial Differential Equations 21}, 1-2 (1996), 179--193.

\bibitem{DP08}
{\sc Danchin, R., and Paicu, M.}
\newblock Existence and uniqueness results for the {B}oussinesq system with
  data in {L}orentz spaces.
\newblock {\em Phys. D 237}, 10-12 (2008), 1444--1460.

\bibitem{AF11}
{\sc de~Almeida, M.~F., and Ferreira, L. C.~F.}
\newblock On the well posedness and large-time behavior for {B}oussinesq
  equations in {M}orrey spaces.
\newblock {\em Differential Integral Equations 24}, 7-8 (2011), 719--742.

\bibitem{FS12}
{\sc Feireisl, E., and Schonbek, M.~E.}
\newblock On the {O}berbeck-{B}oussinesq approximation on unbounded domains.
\newblock In {\em Nonlinear partial differential equations}, vol.~7 of {\em
  Abel Symp.} Springer, Heidelberg, 2012, pp.~131--168.

\bibitem{FV10}
{\sc Ferreira, L. C.~F., and Villamizar-Roa, E.~J.}
\newblock On the stability problem for the {B}oussinesq equations in
  weak-{$L^p$} spaces.
\newblock {\em Commun. Pure Appl. Anal. 9}, 3 (2010), 667--684.

\bibitem{GM89}
{\sc Giga, Y., and Miyakawa, T.}
\newblock Navier-{S}tokes flow in {$\mathbf{R}^3$} with measures as initial
  vorticity and {M}orrey spaces.
\newblock {\em Comm. Partial Differential Equations 14}, 5 (1989), 577--618.

\bibitem{JS14}
{\sc Jia, H., and \v{S}ver\'{a}k, V.}
\newblock Local-in-space estimates near initial time for weak solutions of the
  {N}avier-{S}tokes equations and forward self-similar solutions.
\newblock {\em Invent. Math. 196}, 1 (2014), 233--265.

\bibitem{KP83}
{\sc Kapitanski\u{\i}, L., and Piletskas, K.}
\newblock Spaces of solenoidal vector fields and boundary value problems for
  the {N}avier-{S}tokes equations in domains with noncompact boundaries.
\newblock vol.~159. 1983, pp.~5--36.
\newblock Boundary value problems of mathematical physics, 12.

\bibitem{K99}
{\sc Karch, G.}
\newblock Scaling in nonlinear parabolic equations.
\newblock {\em J. Math. Anal. Appl. 234}, 2 (1999), 534--558.

\bibitem{KP08}
{\sc Karch, G., and Prioux, N.}
\newblock Self-similarity in viscous {B}oussinesq equations.
\newblock {\em Proc. Amer. Math. Soc. 136}, 3 (2008), 879--888.

\bibitem{K92}
{\sc Kato, T.}
\newblock Strong solutions of the {N}avier-{S}tokes equation in {M}orrey
  spaces.
\newblock {\em Bol. Soc. Brasil. Mat. (N.S.) 22}, 2 (1992), 127--155.

\bibitem{KT16}
{\sc Korobkov, M., and Tsai, T.-P.}
\newblock Forward self-similar solutions of the {N}avier-{S}tokes equations in
  the half space.
\newblock {\em Analysis \& PDE 9}, 8 (2016), 1811--1827.


\bibitem{L19}
{\sc Lai, C.-C.}
\newblock Forward discretely self-similar solutions of the MHD equations and the viscoelastic Navier-Stokes equations with damping.
\newblock {\em J. Math. Fluid Mech.~21}, 38 (2019).


\bibitem{LMZ19}
{\sc Lai, B., Miao, C., and Zheng, X.}
\newblock Forward self-similar solutions of the fractional {N}avier-{S}tokes
  equations.
\newblock {\em Adv. Math. 352\/} (2019), 981--1043.

\bibitem{LR02}
{\sc Lemari\'{e}-Rieusset, P.~G.}
\newblock {\em Recent developments in the {N}avier-{S}tokes problem}, vol.~431
  of {\em Chapman \& Hall/CRC Research Notes in Mathematics}.
\newblock Chapman \& Hall/CRC, Boca Raton, FL, 2002.

\bibitem{LR16}
{\sc Lemari\'{e}-Rieusset, P.~G.}
\newblock {\em The {N}avier-{S}tokes problem in the 21st century}.
\newblock CRC Press, Boca Raton, FL, 2016.


\bibitem{T14}
{\sc Tsai, T.-P.}
\newblock Forward discretely self-similar solutions of the {N}avier-{S}tokes
  equations.
\newblock {\em Comm. Math. Phys. 328}, 1 (2014), 29--44.

\bibitem{Tsai-preprint}{\color{blue}
{\sc Tsai, T.-P.}
\newblock Large discretely self-similar solutions to Oberbeck-Boussinesq system with Newtonian gravitational field.
\newblock preprint
\newblock arXiv:2409.14007  [math.AP]}

\bibitem{WK17}
{\sc Wr\'{o}blewska-Kami\'{n}ska, A.}
\newblock The asymptotic analysis of the complete fluid system on a varying
  domain: from the compressible to the incompressible flow.
\newblock {\em SIAM J. Math. Anal. 49}, 5 (2017), 3299--3334.

\bibitem{Y00}
{\sc Yamazaki, M.}
\newblock The {N}avier-{S}tokes equations in the weak-{$L^n$} space with
  time-dependent external force.
\newblock {\em Math. Ann. 317}, 4 (2000), 635--675.


\bibitem{YY24}{\color{blue}
{\sc Yang, Y.}
\newblock Forward self-similar solutions of the MHD-Boussinesq system with Newtonian gravitational field.
\newblock preprint
\newblock 	arXiv:2410.05847 [math.AP]
}

\end{thebibliography}

\end{document}